\documentclass[12pt]{article}
\usepackage{times}
\usepackage{booktabs}
\usepackage{pifont}
\usepackage{caption}
\usepackage{mathrsfs}
\usepackage[fleqn]{amsmath}
\usepackage{amsfonts,amsthm,amssymb,mathrsfs,bbding}
\usepackage{txfonts}
\usepackage{graphics,multicol}
\usepackage{graphicx}
\usepackage{color}
\usepackage{enumerate}
\usepackage{caption}
\usepackage{cite}
\usepackage{latexsym,bm}
\usepackage{indentfirst}
\usepackage{mathtools}
\evensidemargin=\oddsidemargin\topmargin=-1.5cm

\newtheorem{thm}{Theorem}[section]

\newtheorem{lem}{Lemma}[section]
\newtheorem{cor}{Corollary}[section]

\newtheorem{remark}{Remark}
\newtheorem{exam}{Example}

\theoremstyle{definition}

\addtocounter{section}{0}
\begin{document}
\title{Enumeration of cubic Cayley graphs on dihedral groups\footnote{This work is supported
by the National Natural Science Foundation of China (Grant Nos. 11671344, 11261059 and  11531011).}}
\author{{\small Xueyi Huang, \ \ Qiongxiang Huang\footnote{
Corresponding author.}\setcounter{footnote}{-1}\footnote{
\emph{E-mail address:} huangqx@xju.edu.cn.}, \ \ Lu Lu}\\[2mm]\scriptsize
College of Mathematics and Systems Science,
\scriptsize Xinjiang University, Urumqi, Xinjiang 830046, P. R. China}
\date{}
\maketitle
{\flushleft\large\bf Abstract}
Let $p$ be an odd prime, and $D_{2p}=\langle a,b\mid a^p=b^2=1,bab=a^{-1}\rangle$ the dihedral group of order $2p$. In this paper, we completely classify the cubic Cayley graphs on $D_{2p}$ up to isomorphism by means of spectral method. By the way, we show that two cubic Cayley graphs on $D_{2p}$ are isomorphic if and only if they are cospectral. Moreover, we obtain the number of isomorphic classes of cubic Cayley graphs on $D_{2p}$ by using Gauss' celebrated law of quadratic reciprocity.
\vspace{0.1cm}
\begin{flushleft}
\textbf{Keywords:} Cayley graph; dihedral group; cospectral; isomorphic classes; quadratic reciprocity.
\end{flushleft}
\textbf{AMS Classification:} 05C25, 05C50.

\section{Introduction}\label{s-1}
Let $G$ be a finite group, and let $S$ be a subset of $G$ such that $1\not\in S$ and $S$ is symmetric, that is, $S^{-1}=\{s^{-1}\mid s\in S\}$ is equal to $S$. The \emph{Cayley graph} on $G$ with respect to $S$, denoted by $X(G,S)$, is the undirected graph with vertex set $G$ and with an edge $\{g,h\}$ connecting $g$ and $h$ if $hg^{-1}\in S$, or equivalently $gh^{-1}\in S$. In particular, if $G$ is a cyclic group, then  the Cayley graph $X(G,S)$ is  called a \emph{circulant graph}.

Let $X(G,S)$ be the Cayley graph on $G$ with respect to $S$. Suppose that $\sigma\in\mathrm{Aut}(G)$, where $\mathrm{Aut}(G)$ is the \emph{automorphism group} of $G$. Let $T=\sigma(S)$. Then it is easily shown that $\sigma$ induces an isomorphism from $X(G,S)$ to $X(G,T)$. Such an isomorphism is called a \emph{Cayley isomorphism}. However, it is possible for two Cayley graphs $X(G,S)$ and $X(G,T)$ to be isomorphic but no Cayley isomorphisms mapping $S$ to $T$. The Cayley graph $X(G,S)$ is called a \emph{CI-graph} of $G$ if, for any Cayley graph $X(G,T)$, whenever $X(G,S)\cong X(G,T)$ we have $\sigma(S)=T$ for some $\sigma\in\mathrm{Aut}(G)$. A group $G$ is called a \emph{CI-group} if all Cayley graphs on $G$ are CI-graphs. A long-standing open question about Cayley graphs is as follows: which Cayley graphs for a group $G$ are CI-graphs? This question stems from a conjecture proposed by \'{A}d\'{a}m\cite{Adam}: all circulant graphs are CI-graphs of the corresponding cyclic groups. This conjecture was disproved by Elspas and Turner \cite{Elspas}, and however, the conjecture stimulated the investigation of CI-graphs and CI-groups \cite{Babai1,Muzychuk,Muzychuk1,Hirasaka,Dobson,Dobson1,Huang,Huang1,Huang2,Li2}. Another motivation for investigating CI-graphs is to determine the isomorphic classes of Cayley graphs. By the definition, if $X(G,S)$ is a CI-graph, then to decide whether or not $X(G,S)$ is isomorphic to $X(G,T)$, we only need to decide whether or not there exists an automorphism $\sigma\in\mathrm{Aut}(G)$ such that $\sigma(S)=T$. The isomorphic classes of some families of Cayley graphs which are edge-transitive but not arc-transitive were determined in \cite{Li1,Xu1}. For more results about CI-graphs and determination for isomorphic classes of Cayley graphs, one can see  the review paper \cite{Li} and references therein.

The \emph{adjacency spectrum}  of a graph $\Gamma$, denoted by $\mathrm{Spec}(\Gamma)$, is the multiset of eigenvalues of its adjacency matrix. Two graphs are called \emph{cospectral} if they share  the same adjacency spectrum. A graph $\Gamma$ is said to be \emph{determined by its spectrum} (DS for short) if every graph cospectral with it is in fact isomorphic to it.  The question `which graphs are DS?' goes back for about half a century, and originates from chemistry \cite{Gunthard}. In the beginning it was believed that every graph is DS until Collatz and Sinogowitz \cite{Collatz} presented a pair of cospectral trees. In fact, Schwenk \cite{Schwenk} stated that almost all trees are non-DS. Nevertheless, it is strongly believed that almost all graphs are DS. In the past twenty years, the DS problem has aroused a lot of investigation and one can refer to \cite{Dam,Dam1} for surveys.

As Cayley graphs have rich structure properties, considering the DS problem for Cayley graphs is interesting for some authors. To solve the DS problem for Cayley graphs, we first need to consider the following problem:  for a Cayley graph $X(G,S)$, whether or not there exists a Cayley graph on $G$ which is cospectral with $X(G,S)$ but not isomorphic to it. For this purpose, a Cayley graph $X(G,S)$ is said to be \emph{Cay-DS}  if, for any Cayley graph $X(G,T)$, $\mathrm{Spec}(X(G,S))=\mathrm{Spec}(X(G,T))$ implies that $X(G,S)\cong X(G,T)$. Note that a DS Cayley graph is always Cay-DS, but  the converse is not always right. For example, it is known that there are exactly $41$ strongly $14$-regular graphs on $29$ vertices with parameters $(29,14,6,7)$ (see \cite{Colbourn},  p. 856). One of these strongly regular graphs is the Payley graph $P(29)$, which is also a circulant graph \cite{Abdollahi}. Thus $P(29)$ is not DS. However, $P(29)$ is Cay-DS because all circulant graphs of prime order are Cay-DS \cite{Elspas}. Up to now, there are few results about the Cay-DS problem. Elspas and Turner \cite{Elspas} gave some pairs of non-isomorphic cospectral circulant graphs. Babai \cite{Babai} and Abdollahi et. al. \cite{Abdollahi} presented some pairs of non-isomorphic cospectral Cayley graphs on the dihedral group of order $2p$ for any prime $p\geq 13$. Chang and the author \cite{Huang1} proved that a circulant graph whose order is a prime power or the product of two distinct primes is Cay-DS if its generating set satisfies some conditions. However, the Cay-DS problem has far from been resolved.

Let $p$ be an odd prime, and $D_{2p}=\langle a,b\mid a^p=b^2=1,bab=a^{-1}\rangle$ the dihedral group of order $2p$. In this paper, we completely classify the cubic Cayley graphs on $D_{2p}$ up to isomorphism by means of  spectral method, and show that all these graphs are CI-graphs. By the way, we prove that all cubic Cayley graphs on $D_{2p}$ are Cay-DS. Moreover, we obtain the number of isomorphic classes of cubic Cayley graphs on $D_{2p}$ by using Gauss' celebrated law of quadratic reciprocity.

\section{The spectra of Cayley graphs on dihedral groups}
First of all, we recall some basic notions and results of representation theory that will be useful in the subsequent sections.

Let $V$ be a finite dimensional vector space over the field $\mathbb{C}$ of complex numbers. A \emph{representation} of $G$ in $V$ is a  homomorphism $\rho:G\rightarrow GL(V)$, where $GL(V)$ denotes the group of isomorphisms of $V$ onto itself. The dimension of $V$ is called the \emph{degree} of $\rho$. Two representations $\rho:G\rightarrow GL(V)$ and $\varphi:G\rightarrow GL(W)$  are said to be \emph{equivalent}, denoted by  $\rho\sim\varphi$, if there exists an isomorphism $T: V\rightarrow W$ such that $T\rho(g)=\varphi(g)T$ for all $g\in G$.

Let $\rho:G\rightarrow GL(V)$ be a representation. A subspace $W$ of $V$ is \emph{$G$-invariant} if, for all $g\in G$ and $w\in W$, one has $\rho(g)w\in W$. In this case, the restriction of $\rho$ to $W$, i.e.,  $\rho_{|W}:G\rightarrow GL(W)$, is also a representation which is called a \emph{subrepresentation} of $\rho$. If the only $G$-invariant subspaces of $V$ are $\{0\}$ and $V$, then $\rho$ is said to be \emph{irreducible}. It is well known that if $W$ is a $G$-invariant subspace of $V$, then there exists a complement $W^0$ of $W$ in $V$ which is also $G$-invariant. Thus every representation is a direct sum of irreducible representations. The \emph{character} $\chi_{\rho}: G\rightarrow \mathbb{C}$ of $\rho$ is defined by setting $\chi_{\rho}(g)=\mathrm{Tr}(\rho(g))$, where $\mathrm{Tr}(\rho(g))$ is the trace of the representation matrix of $\rho(g)$ with respect to some basis of $V$. Clearly, the degree of $\rho$ equals to $\chi_\rho(1)$, and equivalent representations have the same characters. The character of an irreducible representation is called an \emph{irreducible character}. One can refer to \cite{Jean,Steinberg} for more information about representation theory.

Let $G$ be a finite group. We bulid synthetically a vector space with basis $G$ by setting
$$\mathbb{C}G=\{\sum_{g\in G}c_gg\mid c_g\in \mathbb{C}\}.$$
The \emph{(left) regular representation} of $G$ is the homomorphism $L:G\rightarrow GL(\mathbb{C}G)$ defined by
$$L(g)\sum_{x\in G}c_xx=\sum_{x\in G}c_xgx=\sum_{y\in G}c_{g^{-1}y}y$$
for $g\in G$. The following result is well known.
\begin{lem}[\cite{Jean,Steinberg}]\label{lem-2-1}
Let $L$ be the regular representation of $G$. Then
$$L\sim d_1\rho_1\oplus d_2\rho_2\oplus\cdots\oplus d_s\rho_s,$$
where $\rho_1,\ldots,\rho_s$ are all the non-equivalent irreducible representations of $G$ and $d_i$ is the degree of $\rho_i$ ($1\le i\le s$).
\end{lem}

Let $X(G,S)$ be the Cayley graph on $G$ with respect to $S$, and let $L$ be the regular representation of $G$. For $g\in G$, denote by $R(g)$ the representation matrix of $L(g)$ with respect to the basis $G$ of $\mathbb{C}G$. In \cite{Babai}, Babai expressed the adjacency matrix of $X(G,S)$ in terms of those $R(g)$.
\begin{lem}[\cite{Babai}]\label{lem-2-2}
The adjacency matrix of the Cayley graph $X(G,S)$ is equal to $A=\sum_{s\in S}R(s)$, where $R(s)$ is the representation matrix of $L(s)$ for $s\in S$.
\end{lem}
Denote by $\rho_1,\ldots,\rho_h$ all the non-equivalent irreducible representations of $G$ with degrees $d_1,\ldots,d_h$ ($d_1^2+\cdots+d_h^2=n$), respectively,  and $R_i(g)$ the representation matrix of $\rho_i(g)$ for $g\in G$. By Lemma \ref{lem-2-1}, we have
$$L\sim d_1\rho_1\oplus\cdots\oplus d_h\rho_h,$$
and thus there exists an invertible matrix $P$ such that
$$PR(g)P^{-1}=d_1R_1(g)\oplus d_2R_2(g)\oplus\cdots\oplus d_hR_h(g)$$
for $g\in G$. Therefore, we have
$$PAP^{-1}=\sum_{s\in S}PR(s)P^{-1}=d_1\sum_{s\in S}R_1(s)\oplus\cdots \oplus d_h\sum_{s\in S}R_h(s).$$
According to this equality, Babai \cite{Babai} derived an expression for the spectrum of  the Cayley graph $X(G,S)$ in terms of irreducible characters of $G$.
%Suppose that $\lambda_{i,1},\ldots,\lambda_{i,d_i}$ are all the eigenvalues of the matrix $\sum_{s\in S}R_i(s)$ ($1\le i\le h$), then the spectrum of the Cayley graph $Cay(G,S)$ is given by
%$$Spec(X)=\{[\lambda_{1,1}]^{d_1},\ldots,
%[\lambda_{1,d_1}]^{d_1},\ldots,[\lambda_{h,1}]^{d_h},
%\ldots,[\lambda_{h,d_h}]^{d_h}\}.$$
\begin{lem}[\cite{Babai}]\label{lem-2-3}
The spectrum of the Cayley graph $X(G,S)$ is given by
$$\mathrm{Spec}(X(G,S))=\Big\{[\lambda_{i,k_i}]^{d_i}\mid 1\leq i\leq h, 1\leq k_i\leq d_i\Big\},$$
where
$$\sum_{1\leq k_i\leq d_i}\lambda_{i,k_i}^t=\sum_{s_1,\ldots,s_t\in S}\chi_{\rho_i}\Big(\prod_{k=1}^ts_t\Big)$$
holds for any $t\in \mathbb{N}^*$, and $\chi_{\rho_i}$ is the irreducible character of $\rho_i$ with degree $d_i$ for $1\leq i\leq h$.
\end{lem}
Let $\mathbb{Z}_n$ be the cyclic group of integers module  $n$. It is well  known that the irreducible character $\phi_h$ ($1\leq h\leq n$) of $\mathbb{Z}_n$ is given by
$$\phi_h(k)=e^{\frac{2\pi hk}{n}i}=\cos{\frac{2\pi hk}{n}}+i\sin{\frac{2\pi hk}{n}},~\mbox{where $0\leq k\leq n-1$}.$$
Suppose that $S\subseteq \mathbb{Z}_n\setminus\{0\}$ and $S=-S$. Then the Cayley graph $X(\mathbb{Z}_n,S)$ is a circulant graph. By Lemma \ref{lem-2-3}, one can easily obtain the spectra of circulant grpahs.
\begin{lem}\label{lem-2-4}
Let $\mathbb{Z}_n$ be the cyclic group of integers module  $n$. Suppose that $S\subseteq \mathbb{Z}_n\setminus\{0\}$ and $S=-S$.  Then
the circulant graph $X(\mathbb{Z}_n,S)$ has eigenvalues
$$\lambda_h=\sum_{k\in S}\phi_h(k)=\sum_{k\in I}e^{\frac{2\pi hk}{n}i}=\sum_{k\in I}\cos{\frac{2\pi hk}{n}},$$
where $1\leq h\leq n$.
\end{lem}
Recall that $D_{2n}=\langle a,b\mid a^n=b^2=1,bab=a^{-1}\rangle$ is the dihedral group of order $2n$. In order to determine the spectra of Cayley graphs on $D_{2n}$,  we first need to list the character table of $D_{2n}$.

\begin{lem}[\cite{Jean}]\label{lem-2-5}
Let $D_{2n}=\langle a,b\mid a^n=b^2=1,bab=a^{-1}\rangle$ be the dihedral group of order $2n$. Then the character table of $D_{2n}$ is shown in Tab. \ref{tab-1}.
\end{lem}
\begin{table}[t]\small{
\caption{\label{tab-1}\small{Character table of $D_{2n}$ ($1\le h\le [\frac{n-1}{2}]$).}}
\begin{tabular*}{15cm}{@{\extracolsep{\fill}}cccccc}
\toprule
$n$ is odd &$a^k$&$ba^k$&$n$ is even&$a^k$&$ba^k$\\
 \midrule
$\psi_1$&$1$& $1$&$\psi_1$&$1$& $1$\\
  $\psi_2$&$1$& $-1$&$\psi_2$&$1$& $-1$\\
  $\chi_{h}$&$2\cos{\frac{2\pi hk}{n}}$&$0$&$\psi_3$&$(-1)^k$& $(-1)^k$\\
  --&--&--&$\psi_4$&$(-1)^k$& $(-1)^{k+1}$\\
  --&--&--&$\chi_{h}$&$2\cos{\frac{2\pi hk}{n}}$& 0\\
  \bottomrule
\end{tabular*}}
\end{table}

By Lemma \ref{lem-2-3} and  Lemma \ref{lem-2-5}, we obtain the spectra of Cayley graphs on  $D_{2n}$ immediately.
\begin{thm}\label{thm-2-1}
Let $D_{2n}=\langle a,b\mid a^n=b^2=1,bab=a^{-1}\rangle$ be the dihedral group of order $2n$, and let $S$ be a symmetric subset of $D_{2n}$ such that $1\not\in S$.  Then the Cayley graph $X(D_{2n},S)$ has spectrum $$\mathrm{Spec}(X(D_{2n},S))=\Big\{[\lambda_i]^1;[\mu_{h1}]^2, [\mu_{h2}]^2 \mid 1\leq i\leq 3+(-1)^n; 1\leq h\leq \big[\frac{n-1}{2}\big]\Big\},$$
 where $\lambda_i=\sum_{s\in S}\psi_{i}(s)$ for $1\leq i\leq 3+(-1)^n$ and
$$
\left\{\begin{array}{l}
\mu_{h1}+\mu_{h2}=\sum_{s\in S}\chi_{h}(s)\\
\mu_{h1}^2+\mu_{h2}^2=\sum_{s_1,s_2\in S}\chi_{h}(s_1s_2)
\end{array}\right.
$$
for $1\leq h\leq [\frac{n-1}{2}]$.
\end{thm}

Now we focus on  the cubic Cayley graph $X(D_{2n},S)$, where $|S|=3$. Since $X(D_{2n},S)$ is cubic, we can suppose that $S=\{a^k,a^{-k},ba^i\}$ ($k\neq \frac{n}{2}$ if $n$ is even), or $S=\{ba^{k_1},ba^{k_2},ba^{k_3}\}$, or $S=\{a^{\frac{n}{2}},ba^{i},ba^{j}\}$ ($n$ is even). Then $X(D_{2n},S)$ is said to be of \emph{type-I}, \emph{type-II} and \emph{type-III}  if $S$ possesses the form $\{a^k,a^{-k},ba^i\}$, $\{ba^{k_1},ba^{k_2},ba^{k_3}\}$  and  $\{a^{\frac{n}{2}},ba^{k_1},ba^{k_2}\}$, respectively. From Lemma \ref{lem-2-5} and Theorem \ref{thm-2-1}, we deduce the spectrum of the cubic Cayley graph $X(D_{2n},S)$ by simple computation.
\begin{cor}\label{cor-2-1}
Let $X(D_{2n},S)$ be the cubic Cayley graph on $D_{2n}$ with respect to $S$.
\begin{enumerate}[(1)]
\item If $X(D_{2n},S)$ is of type-I, i.e.,  $S=\{a^k,a^{-k},ba^i\}$, then $X(D_{2n},S)$ has spectrum
    $$\mathrm{Spec}(X(D_{2n},S))=\left\{\begin{array}{ll}
\big\{[3]^1,[1]^1;[2\cos{\frac{2 hk\pi}{n}}\pm1]^2\mid 1\leq h\leq [\frac{n-1}{2}]\big\}& \mbox{if $n$ is odd},\\
\big\{[3]^1,[1]^1,\lambda_3,\lambda_4;[2\cos{\frac{2 hk\pi}{n}}\pm1]^2\mid 1\leq h\leq [\frac{n-1}{2}]\big\}& \mbox{if $n$ is even},
\end{array}\right.
    $$
    where $\lambda_3=2\cdot(-1)^k+(-1)^i$ and $\lambda_4=2\cdot(-1)^k-(-1)^i$.

\item If $X(D_{2n},S)$ is of type-II, i.e.,  $S=\{ba^{k_1},ba^{k_2},ba^{k_3}\}$, then $X(D_{2n},S)$ has spectrum
    $$\mathrm{Spec}(X(D_{2n},S))=\left\{\begin{array}{ll}
\big\{[3]^1,[-3]^1;[\pm \sqrt{a_h(S)}]^2\mid 1\leq h\leq [\frac{n-1}{2}]\big\}& \mbox{if $n$ is odd},\\
\big\{[3]^1,[-3]^1,\lambda_3,\lambda_4;[\pm \sqrt{a_h(S)}]^2\mid 1\leq h\leq [\frac{n-1}{2}]\big\}& \mbox{if $n$ is even},\\
\end{array}\right.
$$
where $\lambda_3,\lambda_4=\pm\big[(-1)^{k_1}+(-1)^{k_2}+(-1)^{k_3}\big]$ and $a_h(S)=3+2\big[\cos\frac{2 h(k_1-k_2)\pi}{n}+\cos\frac{2 h(k_1- k_3)\pi}{n}+\cos\frac{2 h(k_2- k_3)\pi}{n}\big]$.

\item If $X(D_{2n},S)$ is of type-III, i.e.,  $S=\{a^{\frac{n}{2}},ba^{i},ba^{j}\}$, then $X(D_{2n},S)$ has spectrum
    $$\begin{array}{l}
        \mathrm{Spec}(X(D_{2n},S))=
\big\{[3]^1,[-1]^1,\lambda_3,\lambda_4;[(-1)^h\pm\sqrt{2\cos\frac{2 h(i-j)\pi}{n}+2}]^2\mid 1\leq h\leq [\frac{n-1}{2}]\big\}
    \end{array}
$$
where $\lambda_3=(-1)^{\frac{n}{2}}+(-1)^{i}+(-1)^{j}$ and $\lambda_4=(-1)^{\frac{n}{2}}+(-1)^{i}-(-1)^{j}$.
\end{enumerate}
 \end{cor}
\begin{proof}
We only consider the case that $n$ is odd, since the computation is similar when $n$ is even. If $X(D_{2n},S)$ is of type-I, by Lemma \ref{lem-2-5} and Theorem \ref{thm-2-1} we obtain that $\lambda_1=\psi_{1}(a^k)+\psi_{1}(a^{-k})+\psi_{1}(ba^i)=3$, $\lambda_2=\psi_{2}(a^k)+\psi_{2}(a^{-k})+\psi_{2}(ba^i)=1$, and for $1\leq h\leq [\frac{n-1}{2}]$,
$$
\left\{
\begin{array}{ll}
\mu_{h1}+\mu_{h2}=\chi_{h}(a^k)+\chi_{h}(a^{-k})+\chi_{h}(ba^i)=4\cos{\frac{2hk\pi}{n}},\\
\mu_{h1}^2+\mu_{h2}^2=\chi_{h}(a^{2k})+\chi_{h}(a^{-2k})+3\chi_{h}(a^0)=4\cos{\frac{4hk\pi}{n}}+6.\\
\end{array}
\right.
$$
Therefore, we have $\mu_{h1},\mu_{h2}=2\cos{\frac{2hk\pi}{n}}\pm1$.

Similarly, if $X(D_{2n},S)$ is of type-II, then $\lambda_1=\psi_{1}(ba^{k_1})+\psi_{1}(ba^{k_2})+\psi_{1}(ba^{k_3})=3$,
$\lambda_2=\psi_{2}(ba^{k_1})+\psi_{2}(ba^{k_2})+\psi_{2}(ba^{k_3})=-3$, and for $1\leq h\leq [\frac{n-1}{2}]$,
$$
\left\{
\begin{array}{ll}
\mu_{h1}+\mu_{h2}=\chi_{h}(ba^{k_1})+\chi_{h}(ba^{k_2})+\chi_{h}(ba^{k_3})=0,\\
\mu_{h1}^2+\mu_{h2}^2=\sum_{1\le i,j\le 3}\chi_{h}(a^{k_i-k_j})=4\big[\cos{\frac{2h(k_1-k_2)}{n}}+\cos{\frac{2h(k_1-k_3)}{n}}+\cos{\frac{2h(k_2-k_3)}{n}}\big]+6.
\end{array}
\right.
$$
Thus $\mu_{h1},\mu_{h2}=\pm\sqrt{a_h(S)}$, where $a_h(S)=3+2\big[\cos{\frac{2h(k_1-k_2)\pi}{n}}+\cos{\frac{2h(k_1-k_3)\pi}{n}}+\cos{\frac{2h(k_2-k_3)\pi}{n}}\big]$.

We complete this proof.
\end{proof}
The following lemma gives the necessary and sufficient condition for the cubic Cayley graph $X(D_{2n},S)$ to be connected.
\begin{lem}\label{lem-2-6}
Let $X(D_{2n},S)$ be the cubic Cayley graph on $D_{2n}$ with respect to $S$.
\begin{enumerate}[(1)]
\item If $X(D_{2n},S)$ is of type-I, i.e.,  $S=\{a^k,a^{-k},ba^i\}$, then $X(D_{2n},S)$ is connected if and only if
$(k,n)=1$.
\vspace{-0.25cm}
\item If $X(D_{2n},S)$ is of type-II, i.e.,  $S=\{ba^{k_1},ba^{k_2},ba^{k_3}\}$, then $X(D_{2n},S)$ is connected if and only if
$(k_1-k_2,k_1-k_3,k_2-k_3,n)=1$.
\vspace{-0.25cm}
\item If $X(D_{2n},S)$ is of type-III, i.e.,  $S=\{a^{\frac{n}{2}},ba^{i},ba^{j}\}$, then $X(D_{2n},S)$ is connected if and only if
$(i-j,\frac{n}{2})=1$.
\end{enumerate}
\end{lem}
\begin{proof}
(1) and $(3)$ are obvious. Now we prove (2). Suppose that $X(D_{2n},S)$ is connected. Then $\langle S \rangle=D_{2n}$ and so $S$ can generate $\langle a \rangle$. Note that $S=\{ba^{k_1},ba^{k_2},ba^{k_3}\}$ and $ba^{k_i}\cdot ba^{k_j}=a^{k_j-k_i}$, we claim that $(k_1-k_2,k_1-k_3,k_2-k_3,n)=1$. Otherwise, if $(k_1-k_2,k_1-k_3,k_2-k_3,n)=d\neq 1$, then it is easy to see that $\langle S \rangle\cap \langle a\rangle\subseteq\langle a^d \rangle$, which implies that $S$ cannot generate $\langle a \rangle$ because $\langle a^d \rangle\neq \langle a \rangle$ due to $d\neq 1$ and $d\mid n$.  Conversely, if $(k_1-k_2,k_1-k_3,k_2-k_3,n)=1$, then $\{a^{k_1-k_2}, a^{k_2-k_3},a^{k_1-k_3}\}$ can generates $\langle a\rangle$, and thus $\langle S\rangle=D_{2n}$.

The proof is now complete.
\end{proof}
Suppose that $X(D_{2n},S)$ is a connected Cayley graph of type-I with $S=\{a^k,a^{-k},ba^i\}$.  If we put $V(X(D_{2n},S))=V_1\cup V_2$, where $V_1=\langle a\rangle=\{1,a,a^2,\ldots,a^{n-1}\}$ and $V_2=b\langle a\rangle=\{b,ba,ba^2,\ldots,ba^{n-1}\}$,  it is seen that $V_1$ and $V_2$ induce two disjoint cycles of length $n$, which are generated by the two elements $a^k,a^{-k}\in S$ because $(k,n)=1$. Moreover,   the element $ba^i\in S$ connects $a^{j}$ to $ba^{i+j}$ and so generates a perfect matching between $V_1$ and $V_2$. Thus we obtain the following result.

\begin{thm}\label{thm-2-2}
Every connected Cayley graph of type-I is isomorphic to $C_n\square K_2$.
\end{thm}
Let $n$ be an odd number. We see that each cubic Cayley graph on $D_{2n}$ can only  be of type-I or type-II. According to Corollary \ref{cor-2-1},  we claim that those cubic   Cayley graphs on $D_{2n}$ ($n$ is odd) of different types  cannot be cospectral, and so cannot be isomorphic.
\begin{thm}\label{thm-2-3}
Let $n$ be an odd number. Then a  cubic Cayley graph on $D_{2n}$  of type-I cannot be isomorphic to a cubic Cayley graph on $D_{2n}$ of type-II.
\end{thm}

\section{Isomorphic classes  of  cubic Cayley graphs on $D_{2p}$}
In this section, we focus on determining the isomorphic classes of cubic Cayley graphs on the dihedral group $D_{2p}$, where $p$ is an odd prime.

Let $G$ be a group, and let $X(G,S)$ be the Cayley graph on $G$ with respect to $S$.  For any $\sigma\in \mathrm{Aut}(G)$, it is well known that $\sigma$ induces an isomorphism $\Phi_\sigma$ from $X(G,S)$  to $X(G,\sigma(S))$, where $\Phi_\sigma$ is defined by $\Phi_\sigma(g)=\sigma(g)$ for $g\in G$. Such an isomorphism $\Phi_\sigma$ is the  so-called  Cayley isomorphism.

Let $D_{2n}=\langle a,b\mid a^n=b^2=1,bab=a^{-1}\rangle$ be the dihedral group of order $2n$. Then for any $\lambda\in \mathbb{Z}_n^*=\{\lambda\in \mathbb{Z}_n\mid (\lambda,n)=1\}=\mathrm{Aut}(\mathbb{Z}_n)$ and $k\in\mathbb{Z}_n$, we define $\sigma_{\lambda,k}:D_{2n}\longrightarrow D_{2n}$ by setting $\sigma_{\lambda,k}(a^i)=a^{\lambda i}$ and $\sigma_{\lambda,k}(ba^j)=ba^{\lambda j+k}$, where $i,j\in \mathbb{Z}_n$. It is easy to verify that $\sigma_{\lambda,k}\in\mathrm{Aut}(D_{2n})$. Actually, Rotmaler in \cite{Rotmaler} proved that each automorphism of $D_{2n}$ ($n\geq 3$) has this form, or equivalently, $\mathrm{Aut}(D_{2n})=\{\sigma_{\lambda,k}\mid \lambda\in \mathbb{Z}_n^*, k\in \mathbb{Z}_n\}$. From the above arguments, we obtain the following result.
\begin{lem}\label{lem-3-1}
Let $S$ be a subset of $D_{2n}$ satisfying $1\not\in S$ and $S=S^{-1}$. Then for any $\lambda\in \mathbb{Z}_n^*$ and $k\in\mathbb{Z}_n$, we have $X(D_{2n},S)\cong X(D_{2n},\sigma_{\lambda,k}(S))$.
\end{lem}

Let $S$, $T$ be two symmetric subsets of $D_{2n}$ such that $1\not\in S$ and $1\not\in T$. According to Lemma \ref{lem-3-1}, the subsets $S$ and $T$  are said to be \emph{equivalent}, denoted by $S\sim T$, if there exist  $\lambda\in \mathbb{Z}_n^*$ and $k\in \mathbb{Z}_n$ such that $T=\sigma_{\lambda,k}(S)$. It is easy to see that `$\sim$' defines an equivalence relation among the symmetric subsets of $D_{2n}$. The following example shows that there are exactly two isomorphic classes of connected cubic Cayley graphs on  $D_{10}$.

\begin{exam}\label{exam-1}
\emph{Let $D_{10}=\langle a,b\mid a^5=b^2=1,bab=a^{-1}\rangle$ and $S=\{b,ba,ba^2\}\subseteq b\langle a\rangle\subseteq D_{10}$. It is easy to verify that each symmetric subset $T$ ($1\not\in T$ and $|T|=3$) of $b\langle a\rangle$ is equivalent to $S$, which implies that all  connected  Cayley graphs on $D_{10}$ of type-II are isomorphic to $X(D_{10},S)$.  By Theorem \ref{thm-2-2} and Theorem \ref{thm-2-3}, there are exactly two connected cubic Cayley graphs on $D_{10}$ up to isomorphism.}
\end{exam}

Let $D_{2p}=\langle a,b\mid a^p=b^2=1,bab=a^{-1}\rangle$ be the dihedral group of order $2p$ ($p$ is an odd prime). By Lemma \ref{lem-2-6},  every cubic Cayley graphs on $D_{2p}$ is connected because $p$ is a prime. According to Theorem \ref{thm-2-2} and Theorem \ref{thm-2-3}, to determine the isomorphic classes of  cubic Cayley graphs on $D_{2p}$, it suffices to consider those graphs of type-II. For this purpose, we give some useful lemmas.

The following result due to Elspas and Turner \cite{Elspas} shows that every circulant graph of prime order is a CI-graph which is also Cay-DS, here we prefer to give a shorter proof because our result deals with multi-subsets.
\begin{lem}\label{lem-3-2}
Let $p$ be an odd prime, and $\mathbb{Z}_p$ the cyclic group of integers module $p$. Suppose that $S$, $T$ are two symmetric multi-subsets of $\mathbb{Z}_p$ such that $0\not\in S$ and $0\not\in T$. Then the following are equivalent:

\begin{enumerate}[(1)]
\vspace{-0.25cm}
\item $\mathrm{Spec}(X(\mathbb{Z}_p,S))=\mathrm{Spec}(X(\mathbb{Z}_p,T))$.
    \vspace{-0.25cm}
\item There exists some $k\in \mathbb{Z}_p^*$  such that $S=k T$.
    \vspace{-0.25cm}
\item $X(\mathbb{Z}_p,S)\cong X(\mathbb{Z}_p,T)$.
\end{enumerate}
\end{lem}
\begin{proof}
$(1)\Rightarrow(2)$. Suppose that $\mathrm{Spec}(X(\mathbb{Z}_p,S))=\mathrm{Spec}(X(\mathbb{Z}_p,T))$. Then, by Lemma \ref{lem-2-4}, we obtain that
$$\Big\{\lambda_h=\sum_{s\in S}\omega^{hs}\mid 1\leq h\leq p\Big\}=\Big\{\mu_h=\sum_{t\in T}\omega^{ht}\mid 1\leq h\leq p\Big\},$$
where $\omega=e^{\frac{2\pi}{p}i}$.
Therefore, there exists  some $k$ ($1\le k\le p-1$) such that
$$\sum_{s\in S}\omega^{s}=\lambda_1=\mu_k=\sum_{t\in T}\omega^{kt} \mbox{}. $$
Note that $\omega$ is a  primitive $p$-th root of unity and $f(x)=1+x+x^2+\cdots+x^{p-1}$ is the minimal polynomial of $\omega$ with respect to the field $\mathbb{Q}$ of rational numbers. Then $\omega,\omega^2,\ldots,\omega^{p-1}$ form a basis of the extension field $\mathbb{Q}(\omega)$, which can be viewed as a linear space over $\mathbb{Q}$. Hence, from the above equation we claim that $S=kT$ for some $k\in \mathbb{Z}_p^*$.

$(2)\Rightarrow(3)$. Note that $\mathrm{Aut}(\mathbb{Z}_p)=\mathbb{Z}_{p}^*$, by the arguments before Lemma \ref{lem-3-1} we may conclude that  $X(\mathbb{Z}_p,S)\cong X(\mathbb{Z}_p,T)$.

$(3)\Rightarrow(1)$. Obviously.

We complete this proof.
\end{proof}

The following lemma gives a necessary and sufficient condition for two Cayley graphs on $D_{2p}$ of type-II to be cospectral.
\begin{lem}\label{lem-3-3}
Suppose that $X(D_{2p},S)$ and $X(D_{2p},T)$ are two Cayley graphs on $D_{2p}$ of type-II, where $S=\{ba^{s_1},ba^{s_2},ba^{s_3}\}$ and $T=\{ba^{t_1},ba^{t_2},ba^{t_3}\}$. Then $\mathrm{Spec}(X(D_{2p},S)))=\mathrm{Spec}(X(D_{2p},T))$ if and only if
 there exists some $\lambda\in \mathbb{Z}_p^*$ such that $\{\pm(s_1-s_2),\pm(s_1-s_3),\pm(s_2-s_3)\}=\lambda \{\pm(t_1-t_2),\pm(t_1-t_3),\pm(t_2-t_3)\}$.
\end{lem}
\begin{proof}
Suppose that $\mathrm{Spec}(X(D_{2p},S)))=\mathrm{Spec}(X(D_{2p},T))$. By Corollary \ref{cor-2-1}, we have  $\{a_h(S)\mid 1\leq h\leq \frac{p-1}{2}\}=\{a_{l}(T)\mid 1\leq l\leq \frac{p-1}{2}\}$, that is,
\begin{equation}\label{equ-1}
\begin{array}{ll}
&\{\cos\frac{2(s_1-s_2)h \pi}{p}+\cos\frac{2 (s_1-s_3)h\pi}{p}+\cos\frac{2(s_2-s_3)h \pi}{p}\mid 1\leq h\leq \frac{p-1}{2}\}\\
=&\{\cos\frac{2(t_1-t_2) l\pi }{p}+\cos\frac{2(t_1-t_3) l\pi}{p}+\cos\frac{2(t_2-t_3)l \pi}{p}\mid 1\leq l\leq \frac{p-1}{2}\}.
\end{array}
\end{equation}
By simple observation, we may conclude that Eq. (\ref{equ-1}) is equivalent to
\begin{equation}\label{equ-2}
\begin{array}{ll}
&\{\cos\frac{2(s_1-s_2)h \pi}{p}+\cos\frac{2 (s_1-s_3)h\pi}{p}+\cos\frac{2(s_2-s_3)h \pi}{p}\mid 1\leq h\leq p\}\\
=&\{\cos\frac{2(t_1-t_2) l\pi }{p}+\cos\frac{2(t_1-t_3) l\pi}{p}+\cos\frac{2(t_2-t_3)l \pi}{p}\mid 1\leq l\leq p\}.
\end{array}
\end{equation}
Let $A_S=\{s_1-s_2,s_1-s_3,s_2-s_3\}$ and $A_T=\{t_1-t_2,t_1-t_3,t_2-t_3\}$. Thus Eq. (2) implies that $X(\mathbb{Z}_p,A_S\cup(-A_S))$ and $X(\mathbb{Z}_p,A_T\cup(-A_T))$ are cospectral by Lemma \ref{lem-2-4}. Then, by Lemma \ref{lem-3-2},  there exists some $\lambda\in \mathbb{Z}_p^*$ such that $A_S\cup(-A_S)=\lambda [A_T\cup(-A_T)]$, i.e.,  $\{\pm(s_1-s_2),\pm(s_1-s_3),\pm(s_2-s_3)\}=\lambda \{\pm(t_1-t_2),\pm(t_1-t_3),\pm(t_2-t_3)\}$.

Conversely, suppose that $\{\pm(s_1-s_2),\pm(s_1-s_3),\pm(s_2-s_3)\}=\lambda \{\pm(t_1-t_2),\pm(t_1-t_3),\pm(t_2-t_3)\}$ for some $\lambda\in \mathbb{Z}_p^*$. Then $\mathrm{Spec}(X(\mathbb{Z}_p,A_S\cup(-A_S)))=\mathrm{Spec}(X(\mathbb{Z}_p,A_T\cup(-A_T)))$ again by Lemma \ref{lem-3-2}, which leads to Eq. (\ref{equ-2}) again by Lemma \ref{lem-2-4}, and so Eq. (1) holds. Thus $\mathrm{Spec}(X(D_{2p},S)))=\mathrm{Spec}(X(D_{2p},T))$
again by Corollary \ref{cor-2-1}.

It follows our result.
\end{proof}

\begin{lem}\label{lem-3-4}
Let $X(D_{2p},S)$ be the Cayley graph on $D_{2p}$ of type-II with respect to $S=\{ba^{s_1},ba^{s_2},ba^{s_3}\}$. Then $S\sim\{b,ba,ba^{s}\}$ for some $s\in \mathbb{Z}_{p}\setminus\{0,1\}$, and thus $X(D_{2p},S)\cong X(D_{2p},\{b,ba,ba^{s}\})$.
\end{lem}
\begin{proof}
Since $s_1,s_2\in \mathbb{Z}_p$, $s_1\neq s_2$ and $p$ is a prime, we have $s_2-s_1\in \mathbb{Z}_p^*$. Taking $\lambda=(s_2-s_1)^{-1}\in \mathbb{Z}_p^*$ and $k=-(s_2-s_1)^{-1}s_1\in \mathbb{Z}_p$, one can easily  verify that $\sigma_{\lambda,k}(S)=\{b,ba,ba^{s}\}$, where $s=(s_2-s_1)^{-1}(s_3-s_1)\in \mathbb{Z}_{p}\setminus\{0,1\}$. This implies that $S\sim\{b,ba,ba^{s}\}$, and so $X(D_{2p},S)\cong X(D_{2p},\{b,ba,ba^{s}\})$ by Lemma \ref{lem-3-1}.
\end{proof}
Recall that $D_{2p}=\langle a,b\mid a^p=b^2=1,bab=a^{-1}\rangle$ ($p$ is an odd prime) is the dihedral group of order $2p$, and note that each Cayley graph on $D_{2p}$ of type-II corresponds to a subset of $b\langle a\rangle$ with three elements. Let $\mathcal{S}$ be the set consists of all subsets of $b\langle a\rangle$ with three elements. Then `$\sim$' can define an equivalence relation on $\mathcal{S}$. We denote by  $[S]$ the equivalence class of $S\in \mathcal{S}$, i.e., $[S]=\{\sigma_{\lambda,k}(S)\mid \lambda\in\mathbb{Z}_p^*,k\in\mathbb{Z}_p\}$, and by $\mathcal{S}/\!\sim$ the set of equivalence classes. By Lemma \ref{lem-3-4}, for each  $[S]\in \mathcal{S}/\!\sim$, we can choose $S=\{b,ba,ba^s\}$  as the representation element of $[S]$ for some $s\in \mathbb{Z}_{p}\setminus\{0,1\}$. Furthermore, from Lemma \ref{lem-3-1} we know that $X(D_{2p},S')\cong X(D_{2p},S)$ for any $S'\in [S]$.

The following result determines all  equivalence classes of $\mathcal{S}$.
\begin{lem}\label{lem-3-5}
For $s,t\in \mathbb{Z}_{p}\setminus\{0,1\}$, let $S=\{b,ba,ba^s\}$ and  $T=\{b,ba,ba^t\}$. Then $[S]=[T]$ if and only if  $s=t$ or $st=1$ or $s+t=1$ or $s-st-1=0$ or $ t-st-1=0$ or $s+t-st=0$.
\end{lem}
\begin{proof}
Suppose that $[S]=[T]$. Then $S\sim T$, and so there exist some $\lambda\in\mathbb{Z}_p^*$ and $k\in\mathbb{Z}_p$ such that $\sigma_{\lambda,k}(S)=T$. This is equivalent to  $\lambda\{0,1,s\}+k=\{0,1,t\}$ (the equality is taken  modulo  $p$ and we omit it in the following), i.e., $\{k,\lambda+k,\lambda s+k\}=\{0,1,t\}$. If $k=0$, then $\{0,\lambda,\lambda s\}=\{0,1,t\}$, and so $\lambda=1$ and $\lambda s=t$, or $\lambda=t$ amd $\lambda s=1$. In the former case, we have $s=t$, and in the later case, we have $st=1$. If $k=1$, then $\{1,\lambda+1,\lambda s+1\}=\{0,1,t\}$, and so $\lambda+1=0$ and $\lambda s+1=t$, or $\lambda+1=t$ and $\lambda s+1=0$.
The former  implies that $s+t=1$, and the later  implies that $s-st-1=0$. If $k=t$, then $\{t,\lambda+t,\lambda s+t\}=\{0,1,t\}$, and so $\lambda+t=0$ and $\lambda s+t=1$, or $\lambda+t=1$ and $\lambda s+t=0$.
If the former case occurs, then $t-st-1=1$, and if the later case occurs, then $s+t-st=0$.

Conversely, if one of the six conditions holds, we can easily select suitable $\lambda\in\mathbb{Z}_p^*$ and $k\in\mathbb{Z}_p$ such that
$\sigma_{\lambda,k}(S)=T$ by the above arguments. Hence, $S\sim T$, and thus $[S]=[T]$.

This completes the proof.
\end{proof}

Recall that each equivalence class $[S]\in \mathcal{S}/\!\sim$ has a representation element of the form $S=[b,ba,ba^s]$ for some $s\in \mathbb{Z}_p\setminus\{0,1\}$, and all the cubic Cayley graphs on $D_{2p}$ of type-II with equivalent generating sets are isomorphic. Thus the following result determines all  isomorphic classes of cubic Cayley graphs on $D_{2p}$ of type-II.
\begin{thm}\label{thm-3-1}
For $s,t\in\mathbb{Z}_p\setminus\{0,1\}$, let $X(D_{2p},S)$ and $X(D_{2p},T)$ be two Cayley graphs on $D_{2p}$ of type-II with $S=\{b,ba,ba^{s}\}$ and $T=\{b,ba,ba^t\}$. Then the following are equivalent:

\begin{enumerate}[(1)]
    \vspace{-0.25cm}
\item $\mathrm{Spec}(X(D_{2p},S)))=\mathrm{Spec}(X(D_{2p},T))$.
    \vspace{-0.25cm}
\item $s=t$ or $st=1$ or $s+t=1$ or $s-st-1=0$ or $ t-st-1=0$ or $s+t-st=0$.
    \vspace{-0.25cm}
\item $[S]=[T]$.
    \vspace{-0.25cm}
\item $X(D_{2p},S)\cong X(D_{2p},T)$.
\end{enumerate}
\end{thm}
\begin{proof}
It remains to prove $(1)\Rightarrow(2)$ by Lemma \ref{lem-3-5} and Lemma \ref{lem-3-1}.  Suppose that $\mathrm{Spec}(X(D_{2p},S))=\mathrm{Spec}(X(D_{2p},T))$.  By Lemma \ref{lem-3-3}, there exists some $\lambda\in\mathbb{Z}_p^*$  such that
\begin{equation}\label{equ-3}
\{\pm 1,\pm s,\pm(s-1)\}=\lambda\{\pm 1,\pm t,\pm(t-1)\}.
\end{equation}

First assume that all elements of $\{\pm 1,\pm s,\pm(s-1)\}$ are distinct. By simple observation, we only need to consider the following four cases:

{\flushleft\bf Case 1. } $\{1,s,s-1\}=\lambda\{1,t,t-1\}=\{\lambda,\lambda t,\lambda(t-1)\}$;

If $\lambda=1$, then it forces that $s=t$ because $0\neq 2$ due to $p$ is an odd prime.  If $\lambda=s$, then $st=1$ and $s(t-1)=s-1$, or $st=s-1$ and $s(t-1)=1$. The former implies that $s=1$, and the later implies that $0=2$, both of them are impossible. If $\lambda=s-1$, then $(s-1)t=1$ and $(s-1)(t-1)=s$, or $(s-1)t=s$ and $(s-1)(t-1)=1$. The former implies that $s=1$, which is impossible, and the later implies that $s+t-st=0$.

{\flushleft\bf Case 2. } $\{1,s,s-1\}=\lambda\{1,t,-(t-1)\}=\{\lambda,\lambda t,-\lambda(t-1)\}$;

If $\lambda=1$, then we have $s=1$, a contradiction. If $\lambda=s$, then $st=1$ and $-s(t-1)=s-1$, and $st=s-1$ and $-s(t-1)=1$. The former implies that $st=1$, and the later implies that  $s-st-1=0$. If $\lambda=s-1$, then $(s-1)t=1$ and $-(s-1)(t-1)=s$, or $(s-1)t=s$ and $-(s-1)(t-1)=1$. In both cases, we obtain that $0=2$, a contradiction.

{\flushleft\bf Case 3. } $\{1,s,s-1\}=\lambda\{1,-t,t-1\}=\{\lambda,-\lambda t,\lambda(t-1)\}$;

If $\lambda=1$, then we have $s=0$, which is impossible.
If $\lambda=s$, then $-st=1$ and $s(t-1)=s-1$, or $-st=s-1$ and $s(t-1)=1$. In both cases, we obtain that $s=0$, which is impossible. If $\lambda=s-1$, then $-(s-1)t=1$ and $(s-1)(t-1)=s$, or $-(s-1)t=s$ and $(s-1)(t-1)=1$. In both cases, we deduce that $s=0$, a contradiction.

{\flushleft\bf Case 4. } $\{1,s,s-1\}=\lambda\{1,-t,-(t-1)\}=\{\lambda,-\lambda t,-\lambda(t-1)\}$.

If $\lambda=1$, we obtain that $s+t=1$. If $\lambda=s$, then $-st=1$ and $-s(t-1)=s-1$, or $-st=s-1$ and $-s(t-1)=1$. The former implies that $0=2$, and the later implies that $s=1$, both of them are impossible. If $\lambda=s-1$, then $-(s-1)t=1$ and $-(s-1)(t-1)=s$, or $-(s-1)t=s$ and $-(s-1)(t-1)=1$. The former implies that $t-st-1=0$, and the later implies that $s=1$, which is impossible.

Hence, if $\mathrm{Spec}(X(D_{2p},S)))=\mathrm{Spec}(X(D_{2p},T))$, then $s=t$ or $st=1$ or $s+t=1$ or $s-st-1=0$ or $ t-st-1=0$ or $s+t-st=0$.

Next assume that $\{\pm 1,\pm s,\pm(s-1)\}$ contains at least two  elements that are equal. Then $s=-1$, $s=2$ or $s=2^{-1}$ because $s\not\in\{0,1\}$. From Eq. (\ref{equ-3}) we know that $\{\pm 1,\pm t,\pm(t-1)\}$ also contains at least two elements that are equal, similarly we have  $t=-1$, $t=2$ or $t=2^{-1}$. It is easy to verify that these $s$'s and $t$'s satisfy at least one  of the six conditions shown in (2).

We complete this proof.
\end{proof}
Note that every cubic Cayley graph on $D_{2p}$ of type-I corresponds to a symmetric subset of $D_{2p}$ of the form $\{a^s,a^{-s},ba^{i}\}$, and all such subsets are equivalent. In fact, for any two subsets $S=\{a^s,a^{-s},ba^{i}\}$ and $T=\{a^t,a^{-t},ba^{j}\}$, we can take $\lambda=ts^{-1}\in\mathbb{Z}_p^{*}$ and $k=j-ts^{-1}i\in\mathbb{Z}_p$ such that $\sigma_{\lambda,k}(S)=T$, where $\sigma_{\lambda,k}\in\mathrm{Aut}(D_{2p})$. Then, from Theorem \ref{thm-2-2}, Theorem \ref{thm-2-3}, Lemma \ref{lem-3-1}, Lemma \ref{lem-3-4} and Theorem  \ref{thm-3-1}, we deduce the following corollary immediately.
\begin{cor}\label{cor-3-1}
Let $X(D_{2p},S)$ and $X(D_{2p},T)$ be two cubic Cayley graphs on $D_{2p}$. Then the following are equivalent:

\begin{enumerate}[(1)]
    \vspace{-0.25cm}
\item $\mathrm{Spec}(X(D_{2p},S)))=\mathrm{Spec}(X(D_{2p},T))$.
    \vspace{-0.25cm}
\item There exist some $\lambda\in \mathbb{Z}_{2p}^*$ and $k\in \mathbb{Z}_{2p}$ such that $T=\sigma_{\lambda,k}(S)$, where $\sigma_{\lambda,k}\in\mathrm{Aut}(D_{2p})$.
    \vspace{-0.25cm}
\item $X(D_{2p},S)\cong X(D_{2p},T)$.
\end{enumerate}
\end{cor}
\begin{remark}
\emph{It is worth mentioning that Corollary \ref{cor-3-1} implies that all  cubic Cayley graphs on $D_{2p}$ are CI-graphs. In fact, Babai in \cite{Babai1} had shown that $D_{2p}$ is a CI-group, that is, all  Cayley graphs on $D_{2p}$ are CI-graphs. Here we prove the same result for cubic Cayley graphs on $D_{2p}$ by using the spectral method, from which we also know that all cubic Cayley graphs on $D_{2p}$ are Cay-DS}.
\end{remark}
A graph is called \emph{hamiltonian} if it has a spanning cycle. The following corollary shows that every cubic Cayley graph on $D_{2p}$ is hamiltonian.
\begin{cor}\label{cor-3-2}
Let $X(D_{2p},S)$ be a cubic Cayley graph on $D_{2p}$. Then $X(D_{2p},S)$ is hamiltonian.
\end{cor}
\begin{proof}
If $X(D_{2p},S)$ is of type-I, then it is isomorphic to $C_{p}\square K_2$ by Theorem \ref{thm-2-2}, and so is hamiltonian. If $X(D_{2p},S)$ is of type-II, by Lemma \ref{lem-3-4}, we may assume that $S=\{b,ba,ba^s\}$ for some $s\in\mathbb{Z}_p\setminus\{0,1\}$.  It is easy to verify that $X(D_{2p},\{b,ba\})$ is exactly a spanning cycle contained in $X(D_{2p},S)$. Thus $X(D_{2p},S)$ is hamiltonian.
\end{proof}

\section{Enumerating the isomorphic classes of cubic Cayley graphs on $D_{2p}$}
In this section, we will enumerate the isomorphic classes of cubic Cayley graphs on $D_{2p}$. By Theorem \ref{thm-2-2} and Theorem \ref{thm-2-3}, we only need to enumerate the isomorphic classes of those graphs of type-II. For this purpose,  the six conditions in Theorem \ref{thm-3-1} (2) are called \emph{isomorphism conditions} of cubic Cayley graphs on $D_{2p}$ of type-II.

For $s,t\in \mathbb{Z}_p\setminus\{0,1\}$, we say that $s$ and $t$ are \emph{equivalent}, denoted by $s\simeq t$, if $s$ and $t$ satisfy one of the   isomorphism conditions, or equivalently, $t\in\{s,s^{-1},1-s,(s-1)s^{-1},(1-s)^{-1},s(s-1)^{-1}\}$. It is easy to see that `$\simeq$' defines an equivalence relation on $\mathbb{Z}_p\setminus\{0,1\}$. Let $[s]$ denote the equivalence class of $s\in \mathbb{Z}_p\setminus\{0,1\}$. Then we have
\begin{equation}\label{equ-4}
[s]=\{s,s^{-1},1-s,(s-1)s^{-1},(1-s)^{-1},s(s-1)^{-1}\}.
\end{equation}
Moreover, from Theorem \ref{thm-3-1} we see that the number of equivalence classes of $\mathbb{Z}_p\setminus\{0,1\}$, denoted by $\tilde{n}_p$, is just the  number of isomorphic classes of cubic Cayley graphs on $D_{2p}$ of type-II. Thus Theorem \ref{thm-3-1} actually provide us a method to enumerate the number of isomorphic classes of cubic Cayley graphs on $D_{2p}$ of type-II.

Now we give an example to show how to use the above method.
\begin{table}[t]\small{
\caption{\label{tab-2}\small{The equivalence classes of $\mathbb{Z}_p\setminus\{0,1\}$ ($p\leq 23$).}}
\begin{tabular*}{15cm}{@{\extracolsep{\fill}}ccccccc}
\toprule
$\mathbb{Z}_p$&$[2]$&[3]&[4]&[5]&[8]&$\tilde{n}_p$ \\
  \midrule
$\mathbb{Z}_3$&$\{2\}$& --&--&--&--&$1$ \\
$\mathbb{Z}_5$&$\{2,3,4\}$& --&--&--&--&$1$ \\
$\mathbb{Z}_7$&$\{2,4,6\}$&$\{3,5\}$&--&--&--&$2$\\
$\mathbb{Z}_{11}$&$\{2,6,10\}$& $\{3,4,5,7,8,9\}$&--&--&--&$2$ \\
$\mathbb{Z}_{13}$&$\{2,7,12\}$& $\{3,5,6,9,11,8\}$&\{4,10\}&--&--&$3$ \\
$\mathbb{Z}_{17}$&$\{2,9,16\}$& $\{3,6,8,15,12,10\}$&$\{4,5,7,13,14,11\}$&--&--&$3$ \\
$\mathbb{Z}_{19}$&$\{2,10,18\}$&$\{3,7,9,11,13,17\}$&$\{4,5,6,14,15,16\}$&--&$\{8,12\}$&$4$ \\
$\mathbb{Z}_{23}$&$\{2,12,22\}$& $\{3,8,11,13, 16,21\}$&$\{4,6,9, 20,18,15\}$&$\{5,7,10,14,17,19\}$&--&$4$ \\
  \bottomrule
\end{tabular*}}
\end{table}

\begin{exam}
In Tab. \ref{tab-2},  we list all the equivalence classes of $\mathbb{Z}_p\setminus\{0,1\}$ for $p\leq 23$. Let $p=7$. From Eq. (\ref{equ-4}), we see that $\mathbb{Z}_7\setminus\{0,1\}=\{2,3,4,5,6\}$ has two equivalence classes:  $[2]=\{2,4,6\}$, $[3]=\{3,5\}$. Thus there are exactly three cubic  Ceylay graphs on $D_{14}$ up to isomorphism. One  is of type-I: $X(D_{14},\{a,a^{-1},b\})$, and two are of type-II: $X(D_{14},\{b,ba,ba^2\})$, $X(D_{14},\{b,ba,ba^3\})$.
\end{exam}

In order to determine the exact value of $\tilde{n}_p$, we need some basic results about quadratic reciprocity in elementary number theory.

Let $p$ be an odd prime. For $n\in\mathbb{Z}$, the \emph{Legendre symbol} $\left(\frac{n}{p}\right)$ is defined as
$$
\left(\frac{n}{p}\right)=\left\{\begin{array}{ll}
0 & \mbox{ if } p \mbox{ divides } n;\\
1 & \mbox{ if } p \mbox{ does not divide } n \mbox{ and } n \mbox{ is a square modulo } p;\\
-1 & \mbox{ if } p \mbox{ does not divide } n \mbox{ and } n \mbox{ is not a square modulo } p.
\end{array}
\right.
$$
Using the fact that the group of squares of $\mathbb{F}_p^\times$ has index $2$ in $\mathbb{F}_p^\times$, one can easily deduce that (see \cite{Davidoff})
\begin{equation}\label{equ-5}
\left(\frac{mn}{p}\right)=\left(\frac{m}{p}\right)\left(\frac{n}{p}\right)~~~~~(m,n\in \mathbb{Z}).
\end{equation}

The following lemma has its own interest:

\begin{lem}[\cite{Davidoff}, Theorem 2.3.1]\label{lem-4-1}
For $n\in \mathbb{Z}: n^{\frac{p-1}{2}}\equiv\left(\frac{n}{p}\right)~(\mathrm{mod}~p)$.
\end{lem}
Gauss' celebrated \emph{law of quadratic reciprocity} is shown in the following lemma.
\begin{lem}[\cite{Davidoff}, Theorem 2.3.2]\label{lem-4-2}
Let $p$ be an odd prime. Then
\begin{enumerate}[(1)]
\vspace{-0.25cm}
\item $\left(\frac{-1}{p}\right)=(-1)^{\frac{p-1}{2}}$;
    \vspace{-0.25cm}
\item $\left(\frac{2}{p}\right)=(-1)^{\frac{p^2-1}{8}}$;
    \vspace{-0.25cm}
\item if $q$ is an odd prime, distinct from $p: \left(\frac{q}{p}\right)=(-1)^{\frac{(p-1)(q-1)}{4}}\left(\frac{p}{q}\right)$.
\end{enumerate}
\end{lem}
To determine the exact value of $\tilde{n}_p$, the following lemma is critical.
\begin{lem}\label{lem-4-3}
Let $p\ge5$ be a prime. We have
\begin{enumerate}[(1)]
\vspace{-0.25cm}
\item $[2]=\{2, 2^{-1} ,-1\}$;
    \vspace{-0.25cm}
\item if $p\equiv5~(\mathrm{mod}~6)$, then $|[s]|=6$ for all $s\not\in[2]$;
    \vspace{-0.25cm}
\item if $p\equiv1~(\mathrm{mod}~6)$, then the equation $s^2-s+1=0$ has exactly two distinct roots $s_0,s_0'$, where $s_0,s_0'\in \mathbb{Z}_p\setminus\{0,1\}$ and $s_0,s_0'\not\in [2]$.  Furthermore, $[s_0]=\{s_0,s_0'\}$, and $|\tilde{s}|=6$ for all $s\not\in [2]\cup[s_0]$.
\end{enumerate}
\end{lem}
\begin{proof}
If $s=2$,  from Eq. (\ref{equ-4}) we obtain that $[2]=\{2,2^{-1},-1\}$ by simple computation, and thus (1) follows.

In what follows, we assume that $s\in\mathbb{Z}_p\setminus\{0,1\}$ and $s\not\in[2]$. From Eq. (\ref{equ-4}) we  know that
$[s]=\{s,s^{-1},1-s,(s-1)s^{-1},(1-s)^{-1},s(s-1)^{-1}\}$, in which some elements may be equal. For $x,y\in [s]$, we list all the possibilities for $x=y$ in Tab. \ref{tab-3}.
\begin{table}[t]\small{
\caption{\label{tab-3}\small{All the possibilities for $x=y$.}}
\begin{tabular*}{15cm}{@{\extracolsep{\fill}}cccc}
\toprule
   $x\in [s]$   & $y\in [s]$ & $x=y$ & Is possible? \\
  \midrule
  $s$ &$s^{-1}$  &  $s=\pm 1$ & No\\

  $s$ & $1-s$ &  $s=2^{-1}$ & No\\

  $s$ &$(s-1)s^{-1}$ &  $s^2-s+1=0$&Yes\\

  $s$ & $(1-s)^{-1}$ &  $s^2-s+1=0$&Yes\\

  $s$ &  $s(s-1)^{-1}$ & $s=2$&No\\

  $s^{-1}$ &$1-s$  &  $s^2-s+1=0$&Yes\\

  $s^{-1}$ &$(s-1)s^{-1}$ &  $s=1$&No\\

  $s^{-1}$ & $(1-s)^{-1}$ &  $s=2^{-1}$&No\\

  $s^{-1}$ & $s(s-1)^{-1}$ &  $s^2-s+1=0$&Yes\\

  $1-s$ & $(s-1)s^{-1}$ & $s=\pm 1$&No\\

  $1-s$ & $(1-s)^{-1}$ &  $s=0,2$&No\\

  $1-s$ & $s(s-1)^{-1}$ &  $s^2-s+1=0$&Yes\\

  $(s-1)s^{-1}$ & $(1-s)^{-1}$ &  $s^2-s+1=0$&Yes\\

  $(s-1)s^{-1}$ & $s(s-1)^{-1}$ &  $s=2^{-1}$&No\\

 $(1-s)^{-1}$ & $s(s-1)^{-1}$ &  $s=-1$&No\\
  \bottomrule
\end{tabular*}}
\end{table}
From Tab. \ref{tab-3} we see that whether or not $[s]$ contains  equal elements only  depends on  the solutions of the following equation:
\begin{equation}\label{equ-6}
s^2-s+1=0.
\end{equation}
Clearly, Eq. (\ref{equ-6}) is equivalent to
$(2s-1)^2=-3$ because $4\in \mathbb{Z}_p^*$ is invertible. Since $p\nmid (-3)$, by the definition of Legendre symbol we may conclude that Eq. (\ref{equ-6}) has solutions if and only if $\left(\frac{-3}{p}\right)=1$.

If $p\equiv5~(\mathrm{mod}~6)$, then $\left(\frac{p}{3}\right)\equiv p^{\frac{3-1}{2}}\equiv-1~(\mathrm{mod}~3)$ by Lemma \ref{lem-4-1}, and so $\left(\frac{p}{3}\right)=-1$. Thus
$\left(\frac{3}{p}\right)=(-1)^{\frac{(3-1)(p-1)}{4}}\left(\frac{p}{3}\right)=(-1)^{\frac{p+1}{2}}$ by Lemma \ref{lem-4-2} (3). Therefore, by Eq. (\ref{equ-5}) and Lemma \ref{lem-4-2} (1), we get
$\left(\frac{-3}{p}\right)=\left(\frac{-1}{p}\right)\left(\frac{3}{p}\right)=(-1)^{\frac{p-1}{2}}(-1)^{\frac{p+1}{2}}=(-1)^p=-1$. This implies that Eq. (\ref{equ-6}) has no solutions by the above arguments. Hence,  $|[s]|=6$ for all $s\not\in [2]$, and (2) follows.

If $p\equiv1$ (mod $6$), similarly, we obtain that $\left(\frac{p}{3}\right)=1$,  $\left(\frac{3}{p}\right)=(-1)^{\frac{(3-1)(p-1)}{4}}\left(\frac{p}{3}\right)=(-1)^{\frac{p-1}{2}}$, and thus $\left(\frac{-3}{p}\right)=\left(\frac{-1}{p}\right)\left(\frac{3}{p}\right)=(-1)^{\frac{p-1}{2}}(-1)^{\frac{p-1}{2}}=(-1)^{p-1}=1$. This implies that $-3$  is a square modulo  $p$, and so there exists some $x_0\in \mathbb{Z}_p$  such that $x_0^2=-3$. Note that $(-x_0)^2=-3$ and $x_0\neq-x_0$. Thus the equation $x^2=-3$ has exactly two distinct roots $\pm x_0$ in $\mathbb{Z}_p$ because it is a  quadratic equation over the finite field $\mathbb{F}_p=\mathbb{Z}_p$.  It follows that that Eq. (\ref{equ-6}) has exactly two distinct roots: $s_0=2^{-1}(1+x_0)$ and $s_0'=2^{-1}(1-x_0)$.
Clearly, $s_0,s_0'\not\in\{0,1\}$, $s_0'\in [s_0]$ because $s_0s_0'=1$,  and by simple computation  we obtain that $[s_0]=\{s_0,s_0^{-1}=s_0'\}$ and $[s_0]\neq[2]$. Moreover, if $s\not\in [2]\cup [s_0]$,  from  the above arguments  and Tab. \ref{tab-3} we may conclude that all  elements belonging to $[s]$ are distinct, and so $|[s]|=6$. Thus (3) follows.

We complete the proof.
\end{proof}
Note that if $p=3$, then $\tilde{n}_p=1$ (see Tab. \ref{tab-2}). By Lemma \ref{lem-3-5}, we  get the following result immediately.
\begin{thm}\label{thm-4-1}
Let $p$ be an odd prime. Then the number of equivalence classes of $\mathbb{Z}_p\setminus\{0,1\}$ is given by
$$
\tilde{n}_p=\left\{\begin{array}{ll}
1& \mbox{if $p=3$;}\\
\frac{p-5}{6}+1&\mbox{if $p\geq 5$ and $p\equiv 5~(\mathrm{mod}~6)$};\\ \frac{p-1}{6}+1&\mbox{if $p\geq 5$ and $p\equiv 1~(\mathrm{mod}~6)$}.
\end{array}
\right.
$$
\end{thm}
Note that all  cubic Cayley graphs on $D_{2p}$ of type-I are isomorphic (see Theorem \ref{thm-2-2}), and $\tilde{n}_p$  is equal to the  number of isomorphic classes of cubic Cayley graphs on $D_{2p}$ of type-II. By Theorem \ref{thm-4-1}, we obtain the main result of this section immediately.
\begin{thm}
Let $p$ be an odd prime. Then the number  of isomorphic classes of cubic Cayley graphs on $D_{2p}$ is given by
$$
\tilde{N}_p=\left\{\begin{array}{ll}
2& \mbox{if $p=3$;}\\
\frac{p-5}{6}+2&\mbox{if $p\geq 5$ and $p\equiv 5~(\mathrm{mod}~6)$};\\ \frac{p-1}{6}+2&\mbox{if $p\geq 5$ and $p\equiv 1~(\mathrm{mod}~6)$}.
\end{array}
\right.
$$
\end{thm}

\end{document}